\documentclass[12pt]{article}
\usepackage[margin=2.5cm]{geometry}
\usepackage{amsmath,amsthm,amsfonts,amssymb}
\usepackage[mathscr]{euscript}
\usepackage{amsmath}
\usepackage{graphicx,latexsym}
\usepackage{tikz}

\newtheorem{theorem}{Theorem}[section]
\newtheorem{lemma}[theorem]{Lemma}
\newtheorem{proposition}[theorem]{Proposition}

\newtheorem{corollary}[theorem]{Corollary}

\newcommand{\gcovere}{{\rm gcover_{e}}}

\newcommand{\gparte}{{\rm gpart_{e}}}
\newcommand{\BF}{{\rm BF}}
\newcommand{\BN}{{\rm BN}}

\newcommand{\cp}{\,\square\,}

\newcommand{\gp}{{\rm gp}}
\newcommand{\gpe}{{\rm gp_{e}}}
\newcommand{\diam}{{\rm diam}}

\def\cp{\,\square\,}

\begin{document}
\title{Generalization of edge general position problem}

\author{
Paul Manuel$^{a}$
\and
R. Prabha$^{b}$
\and
Sandi Klav\v zar$^{c,d,e}$
}

\date{}

\maketitle
\vspace{-0.8 cm}
\begin{center}
	$^a$ Department of Information Science, College of Computing Science and Engineering, Kuwait University, Kuwait \\
{\tt pauldmanuel@gmail.com}\\
\medskip

$^b$ Department of Mathematics, Ethiraj College for Women, Chennai, Tamilnadu, India \\
{\tt prabha75@gmail.com}\\
\medskip

$^c$ Faculty of Mathematics and Physics, University of Ljubljana, Slovenia\\
{\tt sandi.klavzar@fmf.uni-lj.si}\\
\medskip

$^d$ Faculty of Natural Sciences and Mathematics, University of Maribor, Slovenia\\
\medskip

$^e$ Institute of Mathematics, Physics and Mechanics, Ljubljana, Slovenia\\
\medskip

\end{center}

\begin{abstract}
The  edge geodesic cover problem of a graph $G$ is to find a smallest number of  geodesics that cover the edge set of $G$. The edge $k$-general position problem is introduced as the problem to find a largest set $S$ of edges of $G$ such that no $k-1$ edges of $S$ lie on a common geodesic. We study this dual min-max problems and connect them to an edge geodesic partition problem. Using these  connections, exact values of the edge $k$-general position number is determined for different values of $k$ and for different networks including  torus networks, hypercubes, and Benes networks. 
\end{abstract}


\medskip
\noindent{\bf AMS Subj.\ Class.~(2020)}: 05C12, 05C76

%%%%%%%%%%%%%%%%%%%%%%%%%%%%%%%
\section{Introduction}
%%%%%%%%%%%%%%%%%%%%%%%%%%%%%%%

Dual min-max invariant combinatorial problems are central topics in graph theory and more generally in combinatorics, cf.~\cite{AzBu16}. Here we consider an instance of such dual problems, the edge geodesic cover problem and the edge general position problem, where we will use the first as a tool to study the second. 

An \textit{edge geodesic cover} of $G$ is a set $S$ of geodesics such that each edge of $G$ belongs to at least one geodesic of $S$. The \textit{edge geodesic cover number} of $G$, $\gcovere(G)$, is the minimum cardinality of an edge geodesic cover of $G$. The \textit{edge geodesic cover problem} is to find a minimum cardinality edge geodesic cover of $G$, cf.~\cite{Manuel18}. An \textit{edge geodesic  partition} of $G$ is a set $S$ of geodesics such that each edge of  $G$ belongs to exactly one geodesic of $S$. The \textit{edge geodesic partition number} of $G$, $\gparte(G)$, is the minimum cardinality of an edge geodesic partition of $G$. The \textit{edge geodesic partition problem} is to find a minimum cardinality edge geodesic partition of $G$. A survey on edge geodesic cover and partition problems up to 2018 can be found in~\cite{Manuel18}. 

For $k\ge 3$, we introduce the edge $k$-general position sets as follows. An \textit{edge $k$-general position set} (\textit{edge $k$-gp set} for short) is a set $S$ of edges of $G$ such that no $k$ edges of $S$ lie on a common geodesic, that is, $|S \cap E(P)|\leq k-1$ for any geodesic $P$ of $G$.  An edge $k$-general position set of maximum cardinality in $G$ is called an {\em edge $k$-$\gp$ set}. Iits cardinality is denoted by $k$-$\gpe(G)$ and called the {\em edge $k$-gp number}. An \textit{edge $k$-general position problem} is to find an edge $k$-$\gp$ set. The edge $3$-general position problem is known as the {\em edge general position problem} and was studied for the first time in~\cite{MaPr21}. The corresponding invariant is called the {\em $gp_e$-number} of $G$ and denoted by $\gpe(G)$. The related (vertex) general position problem has already been extensively studied, see~\cite{bijo-2019, klavzar-2021a, klavzar-2021b, manuel-2018a, patkos-2020, tian-2021a, tian-2021b, ullas-2016}. 

The main objective of this paper to demonstrate that the edge geodesic cover problem and the edge $k$-general position problem form a pair of dual min-max combinatorial problems. To do so, we first establish some basic results in Section~\ref{sec:gpe-matching}. The advantage of dual min-max invariant combinatorial problems is that one problem can be solved by means of the other problem. In this paper we apply this approach on the above-mentioned problems. In Section \ref{sec:edge-torus} we determine the edge $k$-gp number for torus graphs $C_{2^r}\cp C_{2^r}$ and $k = 2^t + 1$. Then, in Section~\ref{sec:pc}, we demonstrate that partial cubes contain large edge $k$-gp sets and prove that the edge $k$-gp number of a hypercube $Q_d$ is $ (k-1)2^{d-1}$. In Section~\ref{sec:edge-benes} we solve the edge $k$-gp problem for Benes networks for $k\in \{3, 5\}$. In the rest of the introduction, we give some further definitions needed. 

Let $P_n$ denote the path on $n$ vertices and $C_n$ the cycle on $n$ vertices. The distance $d_G(u, v)$ between vertices $u$ and $v$ of $G$ is the number of edges on a shortest $u, v$-path. Shortest paths are also known as \textit{isometric paths} or \textit{geodesics}. The \textit{diameter} $diam (G)$ of $G$ is the maximum distance between vertices of $G$. A \textit{diametral path} is an isometric path whose length is equal to the diameter of $G$. If $X,Y\subseteq V(G)$, then $d_G(X,Y) = \min_{x\in X, y\in Y} d_G(x,y)$. If $H$ and $H'$ are subgraphs of $G$, then $d_G(H,H') = d_G(V(H),V(H'))$. In this manner, if $e, f\in E(G)$, then $d_G(e,f)$ is the minimum distance between a vertex of $e$ and a vertex of $f$. 

%%%%%%%%%%%%%%%%%%%%%%%%%%%%%%%%
\section{A few preliminary results}
\label{sec:gpe-matching}
%%%%%%%%%%%%%%%%%%%%%%%%%%%%%%%%

In this section we present some preliminary  results on edge $k$-general position sets. We first show that if the diameter of a graph is at most $2k-2$, then the matchings of the graph of cardinality $k$ coincide with edge $k$-general position sets. (Recall that a \textit{matching} of $G$ is a set of independent edges of $G$.)

\begin{proposition}
\label{TMatchEdgek-GenPos}
	Let $G$ be a graph and $k\ge 3$. Then $\diam(G) \le 2k-2$ if and only if every matching of size $k$ is an edge $k$-general position set.
\end{proposition}

\begin{proof}
Suppose that $\diam(G) \le 2k-2$.  Let $M$ be an arbitrary matching of $G$ of order $k$ and assume that the edges from $I$ lie on a common geodesic $P$. Since $M$ is a matching, the length of $P$ is at least $2k-1$ and so $\diam(G) \geq  2k-1$ would hold. As this contradicts our assumption we get that $M$ is an edge $k$-general position set.

Conversely, suppose that every matching of order $k$ is an edge $k$-general position set. Assume on the contrary that $\diam(G) \geq 2k-1$ and let $P$ be a geodesic of length $\diam(G)$. Selecting every second edge of $P$ we construct a matching $M$ of order at least $k$. Let $M'$ be a subset of $M$ with $|M'| = k$. As $M'$ is a matching, by our assumption we have that $M'$ is an edge $k$-general position set, but this is not the case as all the edges from $M'$ lie on $P$. 
\end{proof}

A $j$-geodesic is a geodesic of length $j$. 
The following proposition is a useful tool to prove that a given set of edges is an edge general position set. 

\begin{proposition}
\label{pro:verifying-k-gpe}
Let $S$ be a set of edge-disjoint geodesics in $G$ each of length $j$ and let $\ell = \min_{P,Q\in S} d_G(P,Q)$. If $k\ge 2$ and $d_G(P,Q)  < \ell(k-1) + j(k-2)$ holds for every $P,Q \in S$ such that $P$ and $Q$ lie in a common geodesic, then no $k$ paths from $S$ lie on a common geodesic.
\end{proposition}

\begin{proof}
Suppose on the contrary that the paths $P_1,\ldots, P_k$ from $S$ lie (in this order) on a common geodesic. Then
\begin{align*}
d_G(P_1,P_k) & = \sum_{i=1}^{k-2}(d_G(P_i,P_{i+1}) + j) + d_G(P_{k-1},P_{k}) \\
& = \sum_{i=1}^{k-1}d_G(P_i,P_{i+1}) + j(k-2) \\
& \ge \ell(k-1) + j(k-2),
\end{align*} 
a contradiction since $P_1$ and $P_k$ lie on a common geodesic and we have assumed that then $d_G(P_1,P_k)  < \ell(k-1) + j(k-2)$ holds.
\end{proof}

Setting $j=1$ in Proposition~\ref{pro:verifying-k-gpe} we have the following consequence. 

\begin{corollary}
	\label{cor:verifying-k-gpe1}
Let $S\subseteq E(G)$, $\ell = \min_{e,f\in S}d_G(e,f)$, and $L = \max_{e,f\in S}d_G(e,f)$. If $L  < \ell(k-1) + (k-2)$, then $S$ is an edge $k$-general position set.
\end{corollary}

We conclude the section by the following simple, yet fundamental inequalities comparing $\gpe(G)$, $\gcovere(G)$, and $\gparte(G)$. The result establishes how the edge geodesic cover problem and the edge general position problem constitute dual min-max combinatorial problems.

\begin{lemma}
\label{lem:gpe-gcovere-gparte}
If $G$ is a graph and $k\ge 3$,	then 
$$k{\mbox -}\gpe(G) \leq (k-1) \cdot \gcovere(G) \leq (k-1) \cdot \gparte(G)\,.$$
\end{lemma}

\begin{proof}
Each geodesic from a geodesic cover can contain at most $k-1$ edges from an edge $k$-general position set. Hence the left inequality. The right inequality follows because a geodesic partition is a geodesic cover, cf.~\cite{Manuel18}.  
\end{proof}

%%%%%%%%%%%%%%%%%%%%%%%%%%%%%%%%
\section{The edge $k$-gp problem for torus}
\label{sec:edge-torus}
%%%%%%%%%%%%%%%%%%%%%%%%%%%%%%%%

The {\em Cartesian product} $G\cp H$ of graphs $G$ and $H$ is defined on the vertex set $V(G\cp H) = V(G)\times V(H)$, vertices $(g,h)$ and $(g',h')$ are adjacent if either $gg'\in E(G)$ and $h=h'$, or $g=g'$ and $hh'\in E(H)$, see the book~\cite{imklra-08} for more information on this graph operation.  Cartesian products of cycles $C_n\cp C_m$ are known as {\em torus graphs}. As in this paper we will consider only these products, we simplify the general terminology for products as follows. Two edges $x=(x_1,x_2)$ and $y=(y_1,y_2)$ of a torus is said to be {\em parallel} if $d(x_1,y_2) = d(x_2,y_1)=d(x_1,y_1) +1= d(x_2,y_2) +1$. The edges of $C_n\cp C_m$ that project on the edges of the first factor will be called {\em horizontal edges}, while the edges that project on $C_m$ are {\em vertical edges}. Analogously we will speak of {\em horizontal cycles} (copies of $C_n$ in the product)  and of \textit{vertical cycles}. 

\begin{lemma}
\label{thm:Cycle-gp_e}
If $r\ge 3$ and $2^t \le 2^{r-1} - 2$, then 
$$(2^t+1){\mbox -}\gpe(C_{2^r}) = 2^{t+1}\,.$$
\end{lemma}

\begin{proof}
Set $k = 2^t+1$. Since $\gcovere(C_{2^r}) = \gparte(C_{2^r}) = 2$, Lemma~\ref{lem:gpe-gcovere-gparte} implies that it is enough to show that $k$-$\gpe(C_{2^r}) \geq  2(k-1)$. That is, we need to construct an  edge $k$-general position set $S$ of $C_{2^r}$ with $|S| = 2k - 2 = 2^{t+1}$. We proceed to construct such a set $S$.
	
Let $V(C_{2^r}) = \{v_1,v_2, \ldots, v_{2^r}\}$. Define $S$ to be the set containing edges $e_i=u_{i}v_{i}$, $i\in [2^{t+1}]$, where the edges are equidistant on $C_{2^r}$ , that is, we select them such that  
$$d_{C_{2^r}}(e_i, e_{i+1}) = 2^{r-t-1}-1$$ 
holds for all $i\in [2^{t+1}-1]$, cf.\ Fig~\ref{fig:cycle-gpe}.  

	\begin{figure}[ht!]
		\centering
		\includegraphics[scale=0.55]{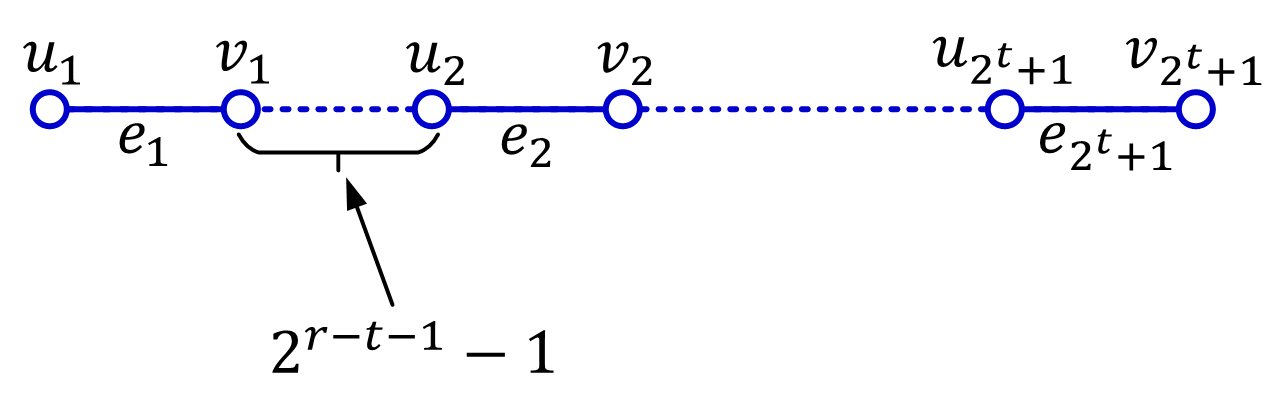}
		\caption{The edges $e_1, e_2 \cdots e_{2^{t}+1}$ of $S'$ lie on geodesic $P$. The end-vertices of the edge $e_i$ are $u_{i}$ and $v_{i}$, such that $u_{1}, v_{1}, u_{2}, v_{2}, \ldots, u_{2^{t}+1}, v_{2^{t}+1}$ are in the increasing order. We have $d(e_i, e_{i+1}) = 2^{r-t-1}-1$.}
		\label{fig:cycle-gpe}
	\end{figure}
	
We claim that $S$ is an edge $k$-general position set. If not, then there exists a subset $S'$ of $S$ such that $|S'| = k = 2^t+1$ and the edges of $S'$ lie on a common geodesic $P$. By the equidistant distribution of the edges from $S$ we may without loss of generality assume that $S' = \{e_1, e_2, \ldots, e_{2^{t}+1}\}$.  As $P$ is a geodesic which contains $2^{t}+1$ edges with the distance $2^{r-t-1}-1$ between the consecutive ones, we have 
$$d_{C_{2^r}}(u_1,v_{2^{t}+1}) = (2^t + 1) + 2^t(2^{r-t-1}-1) = 2^{r-1} + 1\,,$$
a contradiction because $\diam(C_{2^{r-1}}) = 2^{r-1}$.
\end{proof}

Note that Lemma~\ref{thm:Cycle-gp_e} provides an example where both inequalities of Lemma~\ref{lem:gpe-gcovere-gparte} are attained.

\begin{figure}[ht!]
	\centering
	\includegraphics[scale=0.5]{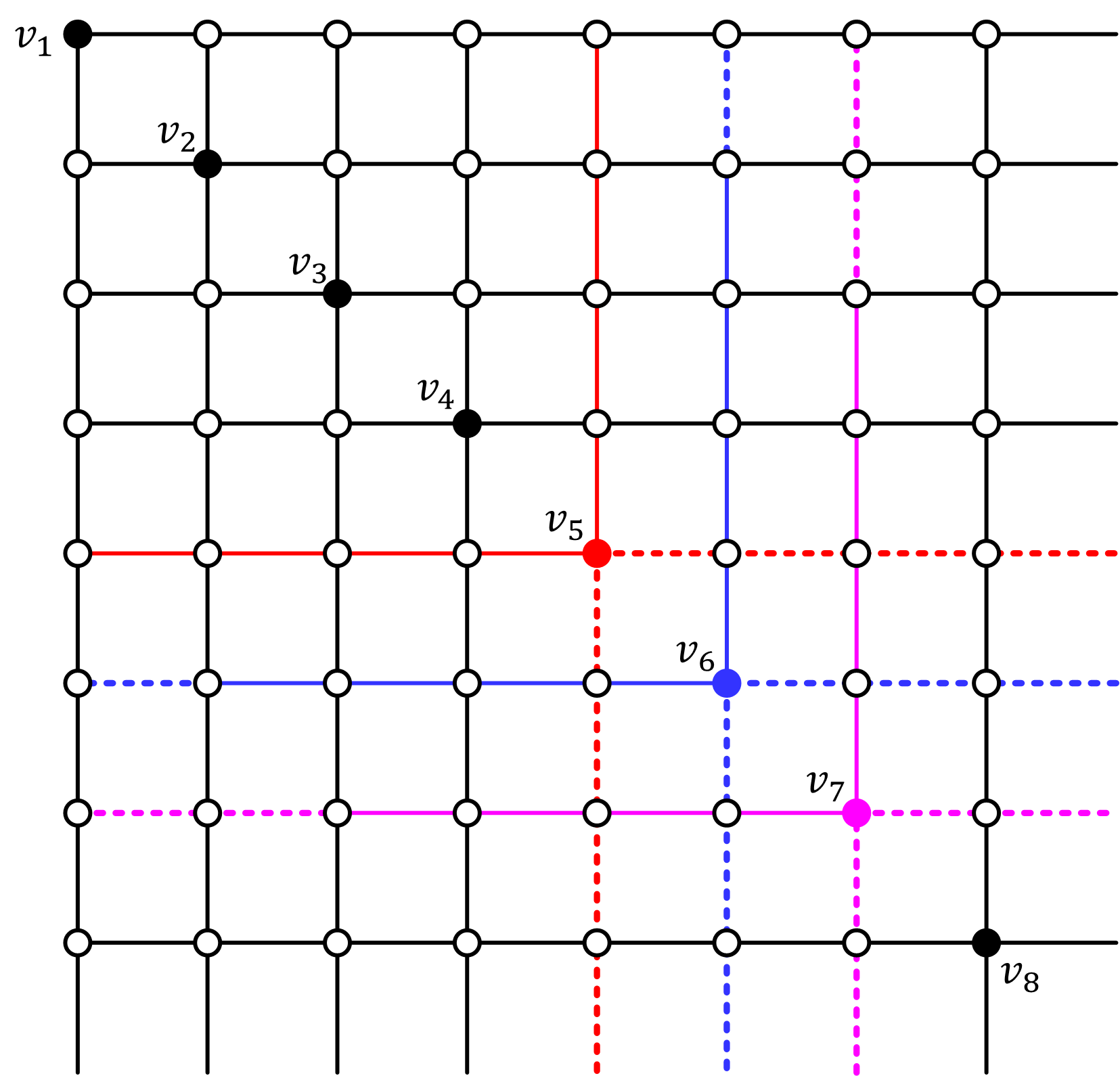}
	\caption{$v_1,\ldots, v_{8}$ are diagonal vertices of $C_8\cp C_8$. The vertex $v_5$ is the red bullet. There are two red lines. One is red solid  line and the other is red dotted  line. Both geodesics (red solid line and  red dotted line) are diametral paths such that $v_5$ is the mid point of both geodesics. The pairs of diametral paths at $v_1,\ldots, v_{8}$ partition $E(G)$.}
	\label{fig:gpe-torus1}
\end{figure}

\begin{proposition}
	\label{prop:torus-gpart_e}
	If $r\ge 2$, then $\gcovere(C_{2r}\cp C_{2r}) =  \gparte(C_{2r}\cp C_{2r}) = 4r$.
\end{proposition}

\begin{proof}
Set $G = C_{2r}\cp C_{2r}$.

Since $\gcovere(G) \geq \left \lceil m(G)/\diam(G) \right \rceil$ (cf.~\cite{Manuel19}), $\diam(G) = 2r$, and $m(G)=2 \cdot 2r \cdot 2r$, we infer that $\gcovere(G) \geq 4r$.

To prove that $\gcovere(G) \leq 4r$ we proceed by construction. Let $v_1, \ldots, v_{2r}$ be the diagonal vertices of $G$ as demonstrated in Fig.~\ref{fig:gpe-torus1}. For each diagonal vertex $v_i$ let $P'_{v_i}$ and $P''_{v_i}$ be two edge disjoint diametral paths as shown in Fig.~\ref{fig:gpe-torus1} for the vertices $v_5$, $v_6$, and $v_7$. Then $v_i$ is the midpoint of $P'_{v_i}$ and $P''_{v_i}$ which is possible because both factors of $G$ are even paths. The set of paths $\{P'_{v_i}$, $P''_{v_i} : i \in [2r]\}$ then partitions $E(G)$. Thus, $\gcovere(G) \le  \gparte(G) \le 4r$. 	
\end{proof}

\begin{figure}[ht!]
	\centering
 \includegraphics[scale=0.34]{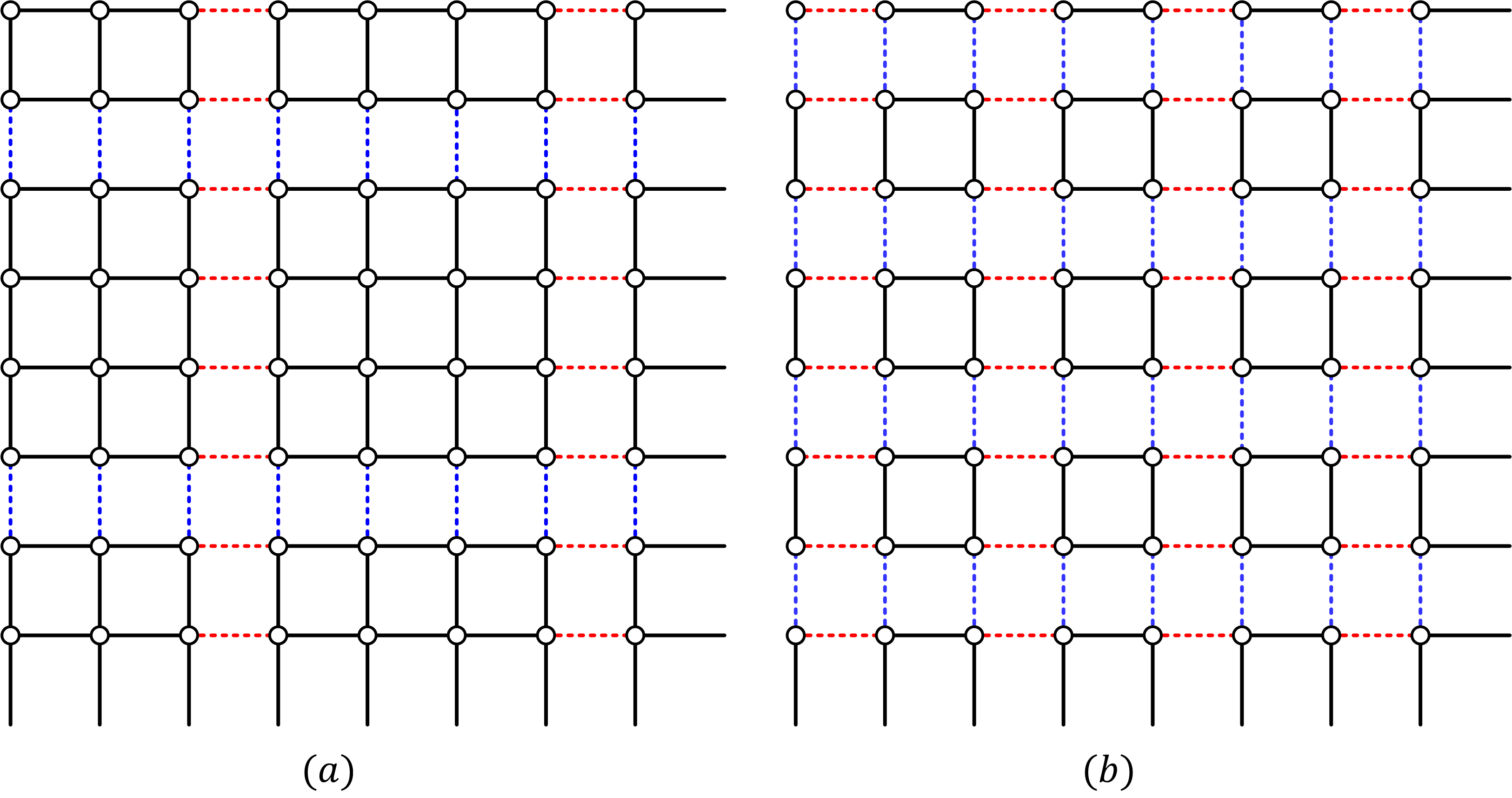}
	\caption{There are two type of edges: red dotted edges and blue dotted edges. The blue dotted edges are vertical and the red dotted edges are horizontal. (a) The red dotted edges and the blue dotted edges form an edge 3-general position set of the $8\times 8$ torus. (b) The red dotted edges and the blue dotted edges form an edge 5-general position set of the $8\times 8$ torus.}
	\label{fig:torus-3_5-gpe}
\end{figure}

\begin{theorem}
\label{thm:torus-gp_e}
If $r\ge 3$ and $2^t \le 2^{r-1} - 2$, then 
$$(2^t+1){\mbox -}\gpe(C_{2^r}\cp C_{2^r}) = 2^{r+t+1}\,.$$ 
\end{theorem}

\begin{proof}
The technique of the proof is parallel to the one from the proof of Theorem~\ref{thm:Cycle-gp_e}. More precisely, we set $k = 2^t+1$ and are going to construct an edge $k$-general position set $S$ of  $G$ with $|S| = (k-1) \cdot \gparte(G) = 2^t \cdot 2^{r+1} =2^{r+t+1}$.
	
	Consider a horizontal cycle of $C_{2^r}\cp C_{2^r}$, say $C^{h}$. Using Theorem \ref{thm:Cycle-gp_e}, construct a set $S^{h}$ of edges from $C^{h}$ of cardinality $2^{t}$ which contains equidistant edges $e^{h}_1$, $e^{h}_2$, $\ldots$, $e^{h}_{2^{t}}$ of the cycle, that is, $d(e^{h}_i, e^{h}_{i+1})  = d(e^{h}_j, e^{h}_{j+1}) = 2^{r-t+1}-1$ for $1 \leq i,j \leq 2^{t}\}$.  Now add all the edges of $S^{h}$ to $S$. Further, add to $S$ all the edges parallel to the edges from $S^{h}$. See Fig.~\ref{fig:torus-3_5-gpe} where the edges from $S$ are the red dotted edges. At this stage we have added to $|S|$ precisely $2^t\cdot 2^r = 2^{r+t}$ edges because there are $2^r$ parallel edges for every given edge in $C_{2^r}\cp C_{2^r}$. We proceed analogously for the vertical cycles of $C_{2^r}\cp C_{2^r}$. That is, we select one such cycle, select in it $2^{t}$ equidistant edges at the distance 	$2^{r-t+1}-1$, and add to $S$ all these edges as well as all the edges parallel to them.  At this stage, the cardinality of $S$ is doubled and thus we have ended up with $|S| = 2^{r+t+1}$. 
		
We claim that the above constructed $S$ is an edge $k$-general position set of $C_{2^r}\cp C_{2^r}$. If this is not the case, then there exists a subset $R$ of $S$ with $|R|=2^t+1$ such that all the edges of $R$ lie on a common geodesic $P$. Since no two parallel edges of a torus lie on a geodesic, we infer that $R$ does not contain any parallel edges. 	Without loss of generality we may assume that $R$ has at least as many horizontal edges than vertical edges. Since $|R|=2^t+1$, this means that $R$ contains at least $2^{t-1}+1$ horizontal edges. Let $R^h$ denote the set of all those horizontal edges, $R^h = \{e^{h}_1, e^{h}_2, \ldots, e^{h}_{s}\}$, where we know that $s \geq 2^{t-1}+1$. Set $e^h_i = u_{i}v_{i}$ and assume without loss of generality that $u_{1}, v_{1}, u_{2}, v_{2}, \ldots, u_s, v_s$ are in a non-decreasing order. (Such an order is possible as no two edges from $R^h$ are parallel.) The situation is illustrated in Fig.~\ref{fig:torus-gpe}. 
	
\begin{figure}[ht!]
\centering
\includegraphics[scale=0.5]{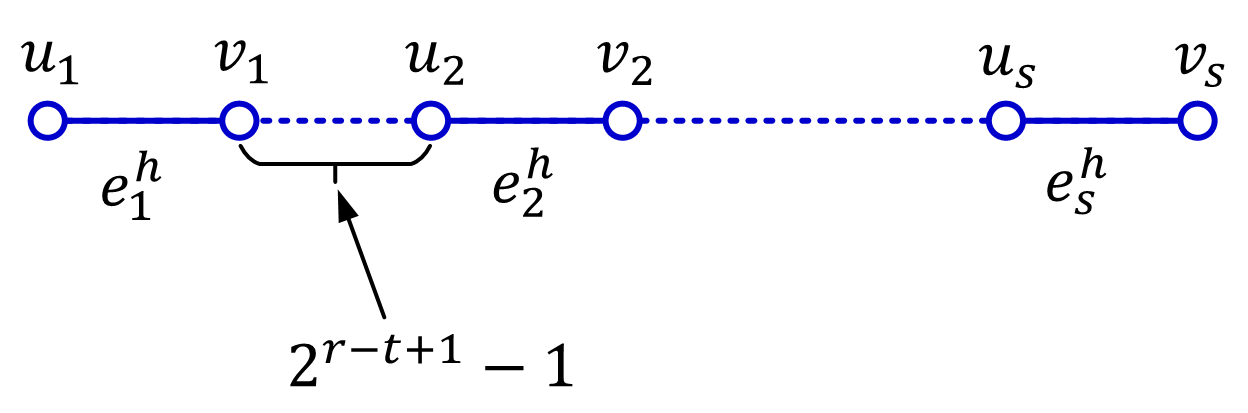}
\caption{$R^h = \{e^{h}_1, e^{h}_2, \ldots, e^{h}_{s}\}$. The vertices $u_{1}, v_{1}, u_{2}, v_{2}, \ldots, u_s, v_s$ are in increasing order. Also, $d(e^h_i, e^h_{i+1}) \ge 2^{r-t+1}-1$.}
		\label{fig:torus-gpe}
	\end{figure}

Consider now the distance $d(u_1,v_s)$ (along the geodesic $P$). We have 
\begin{align*}
d(u_1,v_s) & = (1+d(e^h_1, e^h_2)) +(1+d(e^h_2, e^h_3))  + \cdots + (1+d(e^h_{s-1}, e^h_s))+1 \\
& \ge (1+2^{r-t+1}-1)+(1+2^{r-t+1}-1) + \cdots + (1+ 2^{r-t+1}-1)+1 \\
& =  (s-1) \cdot 2^{r-t+1} + 1 \\
& \geq 2^{t-1} \cdot 2^{r-t+1} + 1 \\
& \geq 2^{r} + 1. 
\end{align*}
This leads to a contradiction because the length of the geodesic $P$ in $C_{2r}\cp C_{2r}$ is greater than $\diam(C_{2r}\cp C_{2r}) = 2^{r}$. This proves the claim and hence the theorem. 
\end{proof}

%%%%%%%%%%%%%%%%%%%%%%%%%%%%%%%
\section{The edge $k$-gp problem for partial cubes}
\label{sec:pc}
%%%%%%%%%%%%%%%%%%%%%%%%%%%%%%%

In this section we extend results on $ \gpe(G)$ of a partial cube $G$ from~\cite{MaPr21} to $k$-$ \gpe(G)$. For this sake, we first recall the concept needed. 

A graph $G$ is a {\em partial cube} if $G$ is a subgraph of some hypercube $Q_d$ such that if $x,y\in V(G)$, then $d_G(x,y) = d_{Q_d}(x,y)$. Papers~\cite{chepoi-2020, marc-2017, mofidi-2020, polat-2021} present a selection of recent developments on partial cubes. Edges $xy$ and $uv$ of a graph $G$ are in relation $\Theta$ if $d_G(x,u) + d_G(y,v) \not= d_G(x,v) + d_G(y,u)$. A connected graph $G$ is a partial cube if and only if $G$ is bipartite and $\Theta$ is transitive~\cite{wink-84}. As $\Theta$ is reflexive and symmetric on an arbitrary graph, it partitions the edge set of a partial cube into $\Theta$-{\em classes}. 

\begin{lemma}
\label{lem:k-1-classes}
Let $G$ be a partial cube, $k\ge 3$, and $F_1, \ldots, F_{k-1}$ be $\Theta$-classes of $G$. Then $\bigcup_{i=1}^{k-1} F_i$ is an edge $k$-gp set of $G$.  
\end{lemma}

\begin{proof}
Consider an arbitrary set $X$ of $k$ edges from $\bigcup_{i=1}^{k-1} F_i$. Then by the pigeonhole principle at least two of the edges from $X$ lie in a common $\Theta$-{\em class} $F_i$. As no two edges from $F_i$ lie on a common geodesic, cf.~\cite[Lemma 11.1]{HIK-2011} we get that the edges from $X$ do not lie on a common geodesic. We conclude that $X$ is an an edge $k$-gp set of $G$. 
\end{proof}

Lemma~\ref{lem:k-1-classes} enables us to find large edge $k$-gp sets in partial cubes. We demonstrate this claim by the following result for hypercubes. Recall that the {\em $d$-dimensional hypercube} $Q_d$, $d\ge 1$, is a graph with $V(Q_d) = \{0, 1\}^r$, and there is an edge between two vertices if and only if they differ in exactly one coordinate. 

\begin{theorem} 
\label{thm:hypecubes}
If $3\le k\le d+1$, then 
$$k{\mbox -}\gpe(Q_d) = (k-1)2^{d-1} =  (k-1) \gcovere(Q_d)  = (k-1) \gparte(Q_d) \,.$$
\end{theorem}

\begin{proof}
It is well-known that $Q_d$ is a partial cube and that it has $\Theta$-classes $F_i$, $i\in [d]$, where $F_i$ is formed by the edges whose endpoints differ in coordinate $i$. Note that $|F_i|  = 2^{d-1}$. Then  Lemma~\ref{lem:k-1-classes} implies that $k{\mbox -}\gpe(Q_r) \ge (k-1)2^{d-1}$. As we have assumed that $k\le d+1$ and $Q_d$ contains $d2^{d-1}$ edges, this is indeed possible.  

To prove the reverse inequality we recall from the proof of~\cite[Theorem 3.2]{MaPr21} that $Q_d$  admits an isometric path edge partition consisting of $2^{d-1}$ paths. Lemma~\ref{lem:gpe-gcovere-gparte} thus implies that 
$$(k-1) \cdot 2^{d-1} \le k{\mbox -}\gpe(Q_d) \leq (k-1) \cdot \gparte(Q_d) \leq (k-1) \cdot 2^{d-1}\,,$$
hence applying Lemma~\ref{lem:gpe-gcovere-gparte} again we have equality everywhere in it for $Q_d$. 
\end{proof}

%%%%%%%%%%%%%%%%%%%%%%%%%%%%%%%%
\section{The edge $k$-gp problem for Benes networks}
\label{sec:edge-benes}
%%%%%%%%%%%%%%%%%%%%%%%%%%%%%%%%

Benes networks were presented in~\cite{Manuel2008} and have been afterwards studied from different contexts, see for instance~\cite{HsuLin08, KaAg14, Kolman02, Leighton1992, NuNaUd20}. For $r\ge 3$, the $r$-dimensional {\em Benes network} $\BN(r)$ is defined as follows. The vertex set consists of all ordered pairs $[s, i]$, where $s$ runs over all $r$-bit binary strings and $i\in \{0,1,\ldots, 2r\}$. In the {\em normal representation} of $BF(r)$, the first coordinate of the vertex is interpreted as the row of the vertex and its second coordinate is a column called {\em level} of the vertex. The vertices $[s, i]$ and $[s', i']$, where $i, i'\le r$, are adjacent if $|i-i'| = 1$, and either $s = s'$ or $s$ and $s'$ differ precisely in the $i^{\rm th}$ bit. When $i, i'\ge r$ the edges are vertically reflected (in the normal representation). The edges between level $i-1$ and level $i$ are called \textit{level $i$ edges} for $1 \leq i \leq 2r$. The formal definition should be clear from Fig.~\ref{fig:Benes-3-dim-nor}, where $\BN(3)$ is shown in its normal representation. Clearly, $\BN(r)$ has $(2r + 1)2^r$ vertices. 

\begin{figure}[ht!]
	\centering
	\includegraphics[scale=0.45]{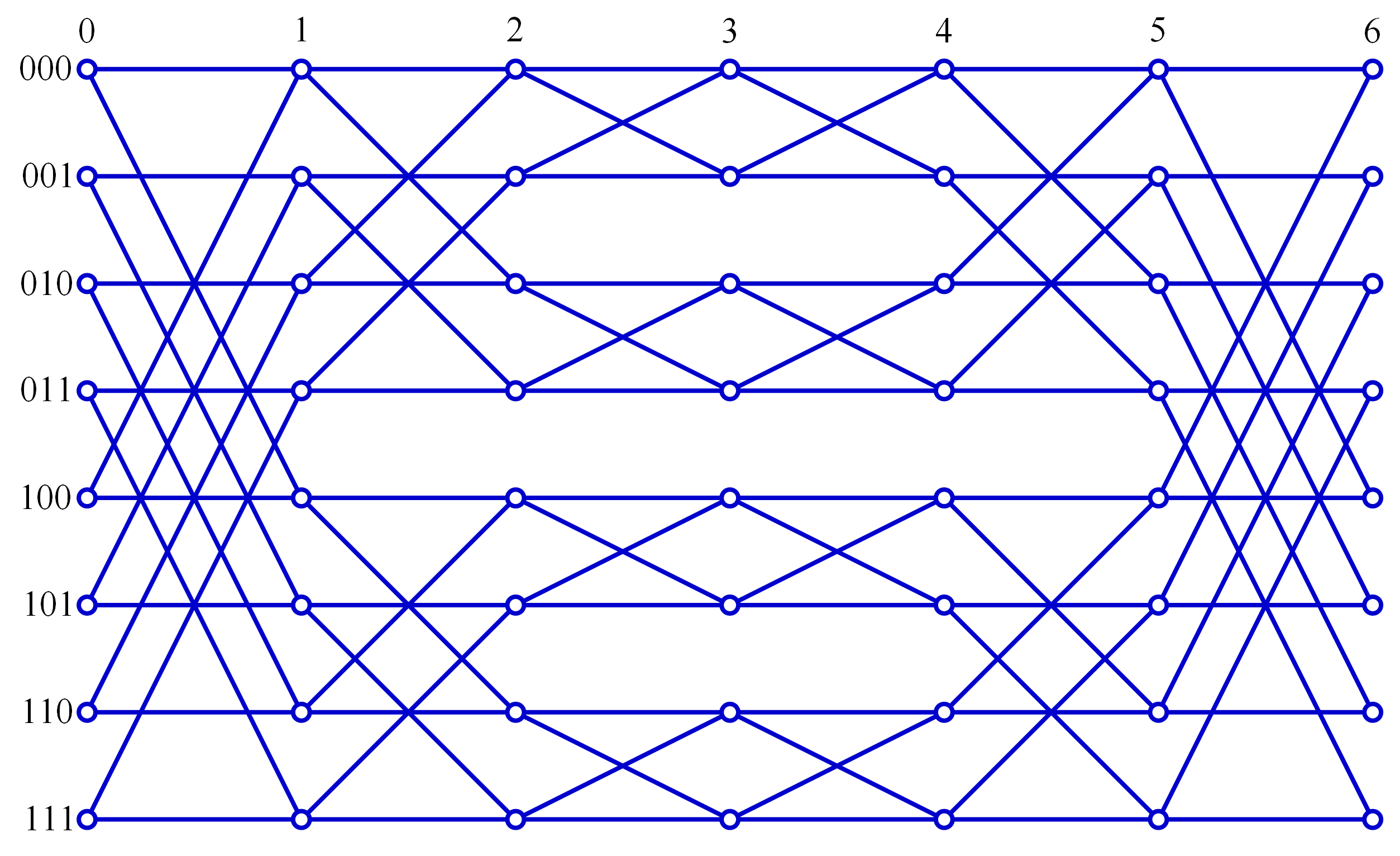}
	\caption{The Benes network $\BN(3)$ consists of back-to-back butterflies and is an edge disjoint union of two butterflies.}
	\label{fig:Benes-3-dim-nor}
\end{figure}

\begin{theorem}
 	\label{thm:BenesGcovere}
 	If $r\ge 3$, then 
 	$$\gcovere(\BN(r)) = \gparte(\BN(r)) = 2^{r+1}\,.$$
\end{theorem}

\begin{proof}
An alternative way to represent Benes networks is that $\BN(r)$ consists of two back-to-back butterflies $\BF(r)$, cf.~\cite{LeMa98,Manuel2008,MaKa2007}, that is, of two copies of $\BF(r)$ sharing level $r$ vertices. It is known that $\diam(\BN(r)) = \diam(\BF(r))= 2r$, cf.~\cite{LeMa98,Manuel2008,MaKa2007}, and that the edge set of $\BF(r)$ can be partitioned with respect to $2^r$ diametral paths~ \cite{Manuel2008}. It follows that the edge set of $\BN(r)$ can be partitioned with respect to $2^{r+1}$ diametral paths of $\BN(r)$. Consequently, Lemma~\ref{lem:gpe-gcovere-gparte} implies $\gcovere(\BN(r)) \le \gparte(\BN(r)) \le 2^{r+1}\,.$
	
To prove the reverse inequality, consider the set $S$ of edges which are incident to the vertices of level $0$, level $r$, and level $2r$. Then $|S| = 2\cdot 2^{r} + 4\cdot 2^{r} + 2\cdot 2^{r} = 2^{r+3}$. Since a geodesic can cover a maximum of four edges of $S$~\cite{Manuel2008}, at least $2^{r+3}/4 = 2^{r+1}$ geodesics are required to cover all the edges of $S$. Hence, $\gcovere(\BN(r)) \geq 2^{r+3}/4 = 2^{r+1}$. Thus, $\gparte(\BN(r)) \geq \gcovere(\BN(r)) \geq 2^{r+1}$.
\end{proof}

\iffalse
\begin{figure}[ht!]
	\centering
	\includegraphics[scale=0.45]{Benes-3-dim-diamond}
	\caption{The diamond representation of $3$-dimensional Benes network, cf.~\cite{Manuel2008,MaKa2007}.}
	\label{fig:Benes-3-dim-diamond}
\end{figure}
\fi

\begin{theorem}
	\label{thm:benes-gpe}
If $k\in \{3,5\}$, then 
$$k{\mbox -}\gpe(BN(r)) = (k-1)\cdot 2^{r+1}\,.$$ 
\end{theorem}

\begin{proof}
By combining Lemma~\ref{lem:gpe-gcovere-gparte} with Theorem~\ref{thm:BenesGcovere}, it is enough to identify an edge general position set $S$ of cardinality $2^{r+2}$ for $k=3$, and an edge $5$-general position set $S'$ of cardinality $2^{r+3}$ for $k=5$. 

For $k=3$, consider the set $S$ of edges in $\BN(r)$ which are incident to the degree $2$ vertices (in levels $0$ and $2r$). Since each level consists of $2^r$ vertices and each vertex at level $0$ and level $2r$ is of degree $2$, $|S| = 2^{r+2}$. An arbitrary geodesic contains at most two edges of $S$, cf.~\cite{Manuel2008}. Thus, $S$ is an edge general position set of $\BN(r)$ of required cardinality.
	
For $k=5$, consider the set of edges $S'$ which are incident to the vertices of level $0$, level $r$, and level $2r$. Then, as already noticed in the proof of Theorem~\ref{thm:BenesGcovere}, $|S'| = 2^{r+3}$. As a geodesic contains at most $4$ edges of $S'$, we conclude that $S$ is an edge general position set of $\BN(r)$ of required cardinality.
\end{proof}

%%%%%%%%%%%%%%%%%%%%%%%%%%%%%%%%%
\section{Conclusion}
\label{sec:conclusion}
%%%%%%%%%%%%%%%%%%%%%%%%%%%%%%%%%
In this paper, we have demonstrated that the edge geodesic cover problem and the edge $k$-general position problem form a pair of dual min-max invariant combinatorial problems. We have solved the edge $k$-general position problem completely for hypercubes and for certain cases of torus. In addition, we have solved the edge $k$-general position problem for Benes networks $\BN(3)$ and $\BN(5)$. The edge geodesic cover problem and the edge geodesic partition problem are completely solved for hypercubes, torus and Benes networks. Studying the interplay between these two concepts seems to be an interesting topic. Among other things, it would be interesting to explore these dual min-max invariant combinatorial problems for intersection graphs, subclasses of perfect graphs, and different Cayley graphs.

%%%%%%%%%%%%%%%%%%%%%%%%%%%%%
\section*{Acknowledgment}
This work was supported and funded by Kuwait University, Kuwait and the Research Project No. is FI 02/21. 
%%%%%%%%%%%%%%%%%%%%%%%%%%%%%


\begin{thebibliography}{99}

\bibitem{bijo-2019} 
  B.S.~Anand, S.V.~Ullas Chandran, M.~Changat, S.~Klav\v{z}ar, E.J.~Thomas, 
  A characterization of general position sets in graphs,
  Appl.\ Math.\ Comput.\ 359 (2019) 84--89.

\bibitem{AzBu16}
Y.~Azar, N.~Buchbinder, H.~Chan, T.-S.H.~Chan, I.~Cohen, A.~Gupta, Z.~Huang, N.~Kang, V.~Nagarajan, J.~Naor, D.~Panigrahi,
Online algorithms for covering and packing problems with convex objectives,
57th Annual IEEE Symposium on Foundations of Computer Science—FOCS 2016, 
148--157, IEEE Computer Soc., Los Alamitos, CA, 2016.   

\bibitem{chepoi-2020}
  V.~Chepoi, K.~Knauer, T.~Marc, 
  Hypercellular graphs: partial cubes without $Q_3$ as partial cube minor,
  Discrete Math.\ 343 (2020) 111678. 

\bibitem{HIK-2011}
  R.~Hammack, W.~Imrich, S.~Klav\v{z}ar,
  Handbook of Product Graphs, Second Edition,
  CRC Press, Boca Raton, FL, 2011.

\bibitem{HsuLin08}
  L.H.~Hsu  and C.K.~Lin,
  Graph Theory and Interconnection Networks, 
  CRC Press,  2008.

\bibitem{imklra-08} 
W.~Imrich, S.~Klav\v{z}ar, D.F.~Rall, 
Topics in Graph Theory: Graphs and Their Cartesian Products, 
A K Peters, Wellesley, 2008.

\bibitem{KaAg14} 
A.~Karimi, K.~Aghakhani, S.E.~ Manavi, F.~Zarafshan, S.A.R.Al-Haddad,
Introduction and analysis of optimal routing algorithm in Benes networks, 
Procedia Comp.\ Sci.\ 42 (2014) 313--319.

\bibitem{klavzar-2021a}
  S.~Klav\v{z}ar, B.~Patk{\'o}s, G.~Rus, I.G.~Yero,
  On general position sets in Cartesian products, 
  Results Math.\ 76 (2021) Article 123. 

\bibitem{klavzar-2021b}
  S.~Klav\v{z}ar, G.~Rus, 
  The general position number of integer lattices,
  Appl.\ Math.\ Comput.\ 390 (2021) Article 125664. 
    
\bibitem{Kolman02} 
P.~Kolman, On non blocking properties of the Benes network,
Lecture Notes Comp.\ Sci.\ 1461 (2002) 259--270.
 
\bibitem{Leighton1992}
F.T.~Leighton, 
Introduction to Parallel Algorithms and Architectures. Arrays, Trees, Hypercubes. Morgan Kaufmann, San Mateo, CA, 1992.

\bibitem{LeMa98}
T.~Leighton, B.~Maggs, R.~Sitaraman, On the fault tolerance of some popular bounded-degree networks, 
SIAM J. Comput.\ 27 (1998) 1303--1333. 

\bibitem{Manuel18}
P.~Manuel,
Revisiting path-type covering and partitioning problems,
arXiv:1807.10613 [math.CO] (25 Jul 2018). 

\bibitem{Manuel19}
P.~Manuel,
On the isometric path partition problem,
Discuss.\ Math.\ Graph Theory 41 (2021) 1077--1089.

\bibitem{Manuel2008}
P.~Manuel, I.M.~Abd-El-Barr, I.~Rajasingh, B.~Rajan,
An efficient representation of {B}enes networks and its applications,
J.\ Discrete Alg.\ 6 (2008) 11--19.

\bibitem{manuel-2018a} 
  P.~Manuel, S.~Klav\v{z}ar, 
  A general position problem in graph theory, 
  Bull.\ Aust.\ Math.\ Soc.\ 98 (2018) 177--187. 

\bibitem{MaKa2007}
P.~Manuel, K.~Qureshi, A.~William, A.~Muthumalai,
VLSI layout of Benes networks,
Journal of Discrete Mathematical Sciences and Cryptography,
10 (2007) 461--472.
https://doi.org/10.1080/09720529.2007.10698132

\bibitem{MaPr21} 
P.~Manuel, R.~Prabha, S.~Klav\v{z}ar, 
The edge general position problem,
Bull.\ Malays.\ Math.\ Sci.\ Soc., (2022) 
https://doi.org/10.1007/s40840-022-01319-8.

\bibitem{marc-2017} 
  T.~Marc, 
  Classification of vertex-transitive cubic partial cubes,
  J.\ Graph Theory 86 (2017) 406--421.

\bibitem{mofidi-2020} 
  A.~Mofidi, 
  On partial cubes, well-graded families and their duals with some applications in graphs,
  Discrete Appl.\ Math.\ 283 (2020) 207--230. 

\bibitem{NuNaUd20} 
M.~Numan, N.~Naz, F.~Uddin, 
New results on topological indices for Benes and butterfly networks, 
U.P.B. Sci. Bull.\ 82 (2020) 33--42.  	

\bibitem{patkos-2020}
  B.~Patk{\'o}s,
  On the general position problem on Kneser graphs,
  Ars Math.\ Contemp.\ 18 (2020) 273--280.
  
\bibitem{polat-2021}
  N.~Polat, 
  On antipodal and diametrical partial cubes,
  Discuss.\ Math.\ Graph Theory 41 (2021) 1127--1145. 

\bibitem{tian-2021a}
  J.~Tian, K.~Xu, 
  The general position number of Cartesian products involving a factor with small diameter,
  Appl.\ Math.\ Comput.\ 403 (2021) Article 126206. 

\bibitem{tian-2021b}
  J.~Tian, K.~Xu, S.~Klav\v{z}ar, 
  The general position number of the Cartesian product of two trees, 
  Bull.\ Aust.\ Math.\ Soc.\ 104 (2021) 1--10.   
  
\bibitem{ullas-2016} 
  S.V.~Ullas Chandran, G.~Jaya Parthasarathy, 
  The geodesic irredundant sets in graphs, 
  Int.\ J.\ Math.\ Combin.\ 4 (2016) 135--143.

\bibitem{wink-84}
  P.~Winkler, 
  Isometric embeddings in products of complete graphs,
  Discrete Appl.\ Math.\ 7 (1984) 221--225.

\end{thebibliography}
\end{document}